\documentclass[11pt,noinfoline]{imsart}

\RequirePackage[OT1]{fontenc}
\RequirePackage{amsthm,amsmath}
\RequirePackage[numbers]{natbib}
\RequirePackage[colorlinks,citecolor=blue,urlcolor=blue]{hyperref}

\usepackage{verbatim}
\usepackage{bbold}
\usepackage{float}

\allowdisplaybreaks

\usepackage{fullpage}
\usepackage{bbm}

\usepackage{layout}

\usepackage{dsfont}
\usepackage[mathscr]{eucal}
\usepackage[toc,page]{appendix}
\usepackage{mathrsfs}
\usepackage{color}
\usepackage{pifont}
\usepackage{bm}
\usepackage{latexsym}
\usepackage{amsfonts}
\usepackage{amssymb}
\usepackage{epsfig}
\usepackage{graphicx}
\usepackage{multirow}

\usepackage{caption}
\usepackage{subcaption}

\newtheorem{theorem}{Theorem}[section]
\newtheorem{corollary}[theorem]{Corollary}
\newtheorem{lemma}[theorem]{Lemma}

\newtheorem{assumption}[theorem]{Assumption}

\usepackage{verbatim}

\usepackage{tikz}
\usetikzlibrary{fit,positioning,arrows,automata,calc}
\tikzset{
main/.style={circle, minimum size = 5mm, thick, draw =black!80, node distance = 10mm},
connect/.style={-latex, thick},
box/.style={rectangle, draw=black!100}
}

\usepackage{mathtools}
\usepackage{tabu}

\hypersetup{
colorlinks=true,       % false: boxed links; true: colored links
linkcolor=blue,        % color of internal links
citecolor=blue,        % color of links to bibliography
filecolor=magenta,     % color of file links
urlcolor=blue        
}
\newcommand*\diff{\mathop{}\!\mathrm{d}}

\newcommand{\E}{\mathbb{E}}

\newcommand\norm[1]{\left\lVert#1\right\rVert}

\begin{document}

\begin{frontmatter}

\title{{\large {Minimax Rate for Optimal Transport Regression Between Distributions}}}

\runtitle{Minimax Rate for Optimal Transport Regression Between Distributions}

\begin{aug}
\author{\fnms{Laya} \snm{Ghodrati}\ead[label=e1]{laya.ghodrati@epfl.ch}} \and
\author{\fnms{Victor M.} \snm{Panaretos}\ead[label=e2]{victor.panaretos@epfl.ch}}

%\thankstext{t1}{Research supported by }

\runauthor{L. Ghodrati \& V.M. Panaretos}

\affiliation{Ecole Polytechnique F\'ed\'erale de Lausanne}

\address{Institut de Math\'ematiques\\
Ecole Polytechnique F\'ed\'erale de Lausanne\\
\printead{e1}, \printead*{e2}}

\end{aug}

\begin{abstract}
    Distribution-on-distribution regression considers the problem of formulating and estimating a regression relationship where both covariate and response are probability distributions. The optimal transport distributional regression model postulates that the conditional Fr\'echet mean of the response distribution is linked to the covariate distribution via an optimal transport map. We establish the minimax rate of estimation of such a regression function, by deriving a lower-bound that matches the convergence rate attained by the Fr\'echet least squares estimator.
\end{abstract}



\end{frontmatter}

{{ \footnotesize
\tableofcontents
}}

\section{Introduction}

Distribution-on-distribution regression considers the the problem of formulating and estimating a regression relationship where both covariate and response are probability distributions. In that sense, it can be seen through the lens of Functional Data Analysis (FDA, \citet{hsing2015theoretical}), and specifically as a special version of the function-on-function regression problem  \citep{morris2015functional,hall2007methodology}. The case of distributions is peculiar because they are bound to satisfy positivity and integral constraints, which are inherently non-linear. Therefore, functional regression methods hinging on the structure of separable Hilbert spaces cannot be directly applied. One needs to either apply a linearising transformation to the regressor/response pairs in order to return to a Hilbert space setting, or to endorse the non-linearity and work in an appropriate native space. See \citet{petersen-review} for a review. In the latter approach, optimal transportation (and the associated Wasserstein spaces) have been seen to offer a canonical geometry for the functional data analysis of distributions \citep{panaretos2019statistical,panaretos2020invitation}. The main two methods so far pursued in this context are those of lifting to the tangent space, postulating a linear regression function between the covariate/response images under the log map \citep{chen2021wasserstein,zhang2022wasserstein}, or to work directly with optimal transport maps, postulating that the response and covariate distributions are related via monotone rearrangement \citep{ghodrati2021distribution}. The two approaches are not directly comparable, though the latter appears to be more simply interpretable while avoiding ill-posedness issues. In particular, under minimal regularity, \cite{ghodrati2021distribution}  establish an upper bound of $N^{-1/3}$ for the rate of estimation of the regression function (with respect to the number $N$ of covariate/response pairs), irrespective of refined regularity properties. The purpose of this note is to establish that $N^{-1/3}$ is indeed the minimax optimal rate, by obtaining a lower bound of same order.

To this aim, we first review the distributional regression model introduced in \cite{ghodrati2021distribution} in more detail. Let $\{(\mu_i,\nu_i)\}_{i=1}^{N}$ be an independent collection of regressor/response pairs in $\mathcal{W}_2(\Omega)\times \mathcal{W}_2(\mathbb{R})$, where the domain $\Omega$ is a compact interval of $\mathbb{R}$. The regression model is
\begin{equation}{\label{model}}
   \nu_{i}=T_{\epsilon_i}\#(T_0\#\mu_i),  \quad  \{\mu_i,\nu_i\}_{i=1}^N,
\end{equation}
where $T_0:\Omega \to \mathbb{R}$ is an unknown optimal map and $\{T_{\epsilon_i}\}_{i=1}^{N}$ is a collection of independent and identically distributed random optimal maps satisfying $E\{T_{\epsilon_i}(x)\}=x$ almost everywhere on $\Omega$. These represent the ``noise" in the model. The regression task is to estimate the unknown $T_0$ from the observations $\{\mu_i,\nu_i\}_{i=1}^N$. 

The probability law induced on $\mathcal{W}_2(\Omega) \times \mathcal{W}_2(\mathbb{R})$ by model \eqref{model} is denoted by $P$. The marginal laws induced on the typical regressor $\mu$ and the typical response $\nu$ are $P_{M}$ and $P_N$, respectively. The (linear) average of $P_M$, i.e. $Q(A)=\int_{\mathcal{W}_2(\Omega)} \mu(A)\diff P_M(\mu)$ is denoted by $Q$. Note that all $\mu$ in the support of $P_M$ are dominated by the measure $Q$, i.e. $\mu \ll Q$ almost surely. Finally, the parameter set of optimal transport maps $\mathcal{T}$ is defined as:
$$\mathcal{T}:=\{T :\Omega \to \Omega: 0\leq T'(x) {<\infty} \text{ for } Q \text{-almost every } x \in \Omega \}.$$

\noindent In this context, the  \cite{ghodrati2021distribution} introduce and study the following Fr\'echet-least-squares estimator:
\begin{equation}{\label{functional}}
  \quad \quad \hat{T}_N:=\arg\min_{T \in \mathcal{T}}   M_N(T),\quad
  M_N(T):= \frac{1}{2N} \sum_{i=1}^N d^2_{\mathcal{W}}(T\# \mu_i,\nu_i).
\end{equation}

Under certain assumptions (see the next Section) they show the $L^2(Q)$ convergence rate of $\hat{T}_N$ to the true map $T_0$ to be $O(N^{-1/3})$. Since there is no ill-conditioning inherent in the setup of Model  \eqref{model}, one might have expected a rate of $O(N^{-1/2})$ when the measures are completely observed (as opposed to being sample from, or observed discretely with error), as is usually the case in functional data analysis. Our purpose is to show that $O(N^{-1/3})$ is indeed the ``right rate'' by establishing a link between Model  \eqref{model} and classical isotonic regression.

\section{Regularity Conditions}

\noindent We now review and (slightly relax) the assumptions made in \cite{ghodrati2021distribution}:
\begin{assumption}{\label{domains}}
The samples are $\{(\mu_i,\nu_i)\}_{i=1}^{N}$  are an i.i.d. collection of regressor/response pairs in $\mathcal{W}_2(\Omega)\times \mathcal{W}_2(\mathbb{R})$, where $\Omega$ is a closed interval of $\mathbb{R}$.
\end{assumption}

\begin{assumption}{\label{noise}}
The error maps $T_\epsilon:\Omega \to \mathbb{R}$ are i.i.d. non-decreasing random maps satisfying $\E(T_{\epsilon_i}(x))=x$ for almost every $x$ on $\Omega$.
\end{assumption}

\begin{assumption}{\label{assumpMaps}}
The regression map $T_0$ is a deterministic element of $\mathcal{T}$.
\end{assumption}

We remark that Assumption \eqref{domains} is weaker than the corresponding assumption in \cite{ghodrati2021distribution} -- namely, we do not require absolute continuity of the input measures. Assumption \eqref{noise} is also weaker relative to the corresponding assumption in \cite{ghodrati2021distribution}, as we do not restrict the range of $T_\epsilon$ to be a compact interval. 

By direct inspection it can be seen that the additional restrictions in \cite{ghodrati2021distribution} are \emph{not} necessary under complete observation of the covariate/response measures. They are made use of in \cite{ghodrati2021distribution} only when one observes the regression/response measures indirectly, e.g. via simple random sampling.

Here we are only concerned with the lower bound (with respect to $N$) when one observes the covariate/response measures completely, as an indicator of the minimax estimation rate intrinsic\footnote{If the covariate/response measures are observed indirectly, additional regularity is asserted on the covariate/response measures in order to be able to recover them. But such assumptions are extrinsic to the structure of the Model \eqref{model} itself.} to Model \ref{model}. In that context, there is no gap between the assumptions used to establish the upper bound $N^{-1/3}$ in \cite{ghodrati2021distribution} and the Assumptions \eqref{domains}, \eqref{noise}, and \eqref{assumpMaps} we use here to derive the lower bound for the convergence rate.

\section{Minimax Rate}
\noindent We now establish the minimax lower bound for the estimation of the map $T_0$. 
\begin{theorem}{\label{lower-bound}}
In the context of Model \ref{model} and under Assumptions \ref{domains}, \ref{noise} and \ref{assumpMaps}, it holds that:
$$R_{N,P}:=\inf_{\hat{T}_N} \sup_{T_0 \in \mathcal{T}} E\bigg\{\norm{\hat{T}_N-T_0}^2_{L^2(Q)}\bigg\}\geq N^{-1/3},$$
where the infimum is taken over all measurable functions of $\{(\mu_i,\nu_i)\}_{i=1}^{N}$ ranging in $\mathcal{T}$.
\end{theorem}

\begin{corollary}
Under the same conditions, the Fr\'echet-least-squares estimator proposed in \cite{ghodrati2021distribution} attains the lower bound in Theorem \ref{lower-bound}, and consequently is minimax optimal.
\end{corollary}

\begin{proof}[Proof of Theorem \ref{lower-bound}]
The idea will be to imbed the setting of isotonic regression within the setting of the current estimation problem.  %To this aim, we will take Dirac predictor measures and choose $T_\epsilon$ so that the pointwise marginal distribution of $T_{\epsilon}(x)$ at a point $x$ be Gaussian with mean $x$ and some fixed variance. The proof will then follow the usual path for establishing the isotonic rate.
We will then use Fano's method (Theorem 15.2. and Lemma 15.5 from \cite{wainwright2019high})as restated in the Appendix for our purposes, following the usual path for establishing the isotonic rate.

First and without loss of generality, we assume that $\Omega = [0,1]$. Suppose $P_M$ is supported on the set of measures $S:=\{\delta_x \text{ s.t. } x\in [0,1]\}$, where $\delta_x$ is a point mass at $x\in[0,1]$. Suppose also that $\diff P_M(\delta_x)=\diff p(x)$, where $p$ is a distribution on $[0,1]$ with bounded density. Note that in this setting, we can see that the distribution $p$  is equal to $Q$ (defined in the first section). 

Further let $\sigma^2>0$ and suppose that given $x\in [0,1]$ the marginal distribution of the real-valued random variable $T_\epsilon(x)$ is centered Gaussian with variance $\sigma^2$, i.e.
$$T_\epsilon(x) \sim N(x,\sigma^2),\qquad \forall x\in [0,1].$$
To see that such family of random maps exists, take each random map to be $T_\epsilon(x)=I(x)+\sigma Z$ where $I(x)=x$ is the identity map and $Z\sim N(0,1)$ is a standard Gaussian. By construction such maps are increasing and their marginal distribution at any fixed point is a Gaussian.

In the setting we have constructed, both predictor and response distributions are supported on a single point (Dirac measures). We can thus conveniently represent them by identifying them with their singleton support. More precisely we represent each pair of predictor/response distributions $(\mu_i,\nu_i)$ via their support  $(X_i,Y_i)$. Therefore, we assume that we observe the collection $\{(X_i,Y_i)\}_{i=1}^N$, where $X_i \in [0,1]$ and are i.i.d. samples from distribution $p$, and $Y_i$ are i.i.d. samples from the distribution $N(T_0(X_i),\sigma^2)$, i.e. the marginal distribution of $Y_i$ given $X_i = x$ is $N(T_0(x),\sigma^2)$. The estimation of the true map $T_0$ in this setting is now equivalent to the estimation of an isotonic regression map from the sample pairs $\{(X_i,Y_i)\}_{i=1}^N$.

Note that any map $T$, induces a probability distribution $P_T(X,Y)\in \mathcal{P}(\mathbb{R}^2)$. Let $\mathbb{P}_T$ denote the following family of distributions on $\mathbb{R}^2$
$$\mathbb{P}_T= \{P_T, s.t. \; T\in \mathcal{T}\}.$$
We want to find an upper-bound for the $\epsilon$-covering number of $\mathbb{P}_T$ in the square root $KL$ divergence, denoted by $N_{KL}(\epsilon;\mathbb{P}_T)$. Since for all $T\in\mathcal{T}$ we can write $P_T(X,Y) =p(X)P_T(Y|X)$, we only need to control the $\epsilon$-covering number of the conditional distributions $P_T(Y|X)$.

The idea is to show $N_{KL}(\epsilon;\mathbb{P}_T)$ can be upper-bounded by the bracketing entropy of the set $\mathcal{T}$. First note that according  to  \cite[Thm 2.7.5]{van1996weak}, we have the following upper-bound for the bracketing entropy of the set  $\mathcal{T}$:
$$\log N_{[]}(\epsilon,\norm{.}_{L^2(Q)},\mathcal{F})\leq K\left(\frac{1}{\epsilon}\right).$$
Since $P_T(Y|X)$ is a Gaussian distribution, for any two maps $T_1$ and $T_2$, we can control
$$KL(P_{T_1}||P_{T_2})\leq \frac{1}{2}\norm{T_1-T_2}^2_{L^2(p)}=\frac{1}{2}\norm{T_1-T_2}^2_{L^2(Q)},$$
 Therefore we conclude that $\log N_{KL}(\epsilon;\mathbb{P}_T)$ is no larger than $\log N_{[]}(\mathcal{T},\epsilon, \norm{.}_{L^2(Q)})\leq \epsilon^{-1}$. 

Now we can take any $\delta$-packing on the set $\mathcal{T}$. We know $\log M(\mathcal{T},\delta, \norm{.}_{L^2(Q)}) \asymp \frac{1}{\delta}$, where $M(\mathcal{T},\delta, \norm{.}_{L^2(Q)})$ is the $\delta$-packing number of the set $\mathcal{T}$. Take $\Phi(\delta)=\delta$, then using Theorem \ref{fano's ineq} and Lemma \ref{KL-epsilon-bound} (Appendix) we can write
$$ \mathfrak{M}(\theta(\mathcal{P});\rho)\geq \frac{\delta}{2}\bigg(1-\frac{\log N_{[]}(\mathcal{T},\epsilon, \norm{.}_{L^2(p)})+\epsilon^2 +\log 2}{\log M(\mathcal{T},\delta, \norm{.}_p)}\bigg).$$
Finally choosing $\epsilon \asymp \delta \asymp N^{-1/3}$ yields the desired rate.

\end{proof}

\section{Appendix: Fano's Method}

In this section, we restate the Fano's method in the format that we use to prove the theorem \ref{lower-bound}, which is taken from \cite{wainwright2019high}.

Given a class of distributions $\mathcal{\mathbb{P}}$, we let $\theta$ denote a functional on the space $\mathcal{P}$ that is a mapping from a distribution $\mathbb{P}$ to a parameter $\theta(\mathbb{P})$ taking values on some space $\Omega$. Let $\rho: \Omega \times \Omega \to [0,\infty)$ be a given metric. Also let $\Phi:[0,\infty]\to [0,\infty)$ be an increasing function. Then we define the $\rho$-minimax risk for the estimation of $\theta$ as:
$$\mathfrak{M}(\theta(\mathcal{P});\Phi\circ \rho):=\inf_{\hat{\theta}}\sup_{\mathbb{P}\in\mathcal{P}} E\big[\Phi\big(\rho(\hat{\theta},\theta(\mathbb{P}))\big)\big].$$

The following theorem (proposition 15.2 \cite{wainwright2019high}) gives a lower bound on the minimax error.
\begin{theorem}{\label{fano}}(Generalized Fano's inequality)
   Let $\{\theta^1,\cdots,\theta^M\}$ be a $2\delta$-separated set in the $\rho$ semi-metric on $\Theta(\mathcal{P})$, and suppose that $J$ is uniformly distributed over the index set $\{1,\cdots,M\}$, and $(Z|J=j)\sim P_{\theta^j}$. Then for any increasing function $\Phi:[0,\infty]\to [0,\infty)$, the minimax risk is lower bounded as
   \begin{equation}\label{fano's ineq}
       \mathfrak{M}(\theta(\mathcal{P});\Phi\circ \rho)\geq \Phi(\delta) \bigg\{1-\frac{I(Z;J)+\log 2}{\log M}\bigg\},
   \end{equation}
where $I(Z;J)$ is the mutual information between $Z$ and $J$.
\end{theorem}

In order to find an upper-bound for the mutual information in the inequality \ref{fano's ineq}, we use lemma 15.5 of \cite{wainwright2019high} which we restate here:

\begin{lemma}(Yang-Barron method)
Let $N_{KL}(\epsilon;\mathcal{P})$ denote the $\epsilon$-covering number of $\mathcal{P}$ in the square-root KL divergence. Then the mutual information is upper bounded as

\end{lemma}
\begin{equation}{\label{KL-epsilon-bound}}
    I(Z;J)\leq \inf_{\epsilon>0} \{\epsilon^2 +\log N_{KL}(\epsilon;\mathcal{P})\}.
\end{equation}

\bibliographystyle{imsart-nameyear}
\bibliography{paper}

\end{document}